\documentclass{amsart}
\usepackage{amssymb}
\usepackage{amsmath}
\usepackage[dvipdfm,colorlinks,backref,pagebackref,hypertexnames=false,linkcolor=blue,urlcolor=green,citecolor=red]{hyperref}
\newcommand{\eproof}{\hfill $\Box$}

\newcommand{\N}{{\mathbb{N}}}

\newcommand{\Stab}{{\rm Stab}}

\newcommand{\PSL}{{\rm PSL}}
\newcommand{\PGL}{{\rm PGL}}

\newcommand{\ord}{{\rm o}}
\newcommand{\Aut}{{\rm Aut}}

\newcommand{\Id}{{\rm Id}}
\newcommand{\Syl}{{\rm Syl}}

\newcommand{\la}{\langle}
\newcommand{\ra}{\rangle}

\newcommand{\td}{\widetilde}

\newtheorem{theorem}{Theorem}[section]
\newtheorem{lemma}[theorem]{Lemma}

\newtheorem{corollary}[theorem]{Corollary}
\newtheorem{definition}[theorem]{Definition}
{ \theoremstyle{remark} \newtheorem{remark}[theorem]{Remark}}
{ \theoremstyle{remark} \newtheorem{example}[theorem]{Example} }

\newcommand{\cal}{\mathcal}
 \newcommand{\CB}{\mbox{$\cal B$}}
 \newcommand{\CD}{\mbox{$\cal D$}}

 \newcommand{\CP}{\mbox{$\cal P$}}

 \newcommand{\CV}{\mbox{$\cal V$}}

\newcommand{\AT}{\mathbb{ATLAS}}

\begin{document}

\title[On automorphisms of designs] {On the automorphisms of designs constructed from finite simple groups}

\author{Tung Le and Jamshid Moori}

\address{T.L.: %School of Mathematical Science,
North-West University, Mafikeng, South Africa}

\address{J.M.: %School of Mathematical Science,
North-West University, Mafikeng, South Africa}

\email{T.L.: Lttung96@yahoo.com}

\email{J.M.: Jamshid.Moori@nwu.ac.za}

\date{\today}
%27Dec2012

\begin{abstract}
Here we study the automorphism groups of $1$-designs constructed from finite nonabelian simple groups by using two methods presented in Moori (Information Security, Coding Theory and Related Combinatorics, 2011) \cite{Moori}. We obtain some general results for both and improve one of these methods.
In an application to the sporadic Mathieu groups $M_{n}$, we are able to retrieve the Steiner systems $S(t,t+3,n)$ where $(n,t)\in\{(22,3),(23,4),(24,5)\}$.
\end{abstract}



\maketitle

\section{Introduction}
It is well-known that $1$-designs have important applications in codding theory. Two methods for constructing designs from finite nonabelian simple groups were introduced by Moori and Key (see \cite{Moori}). Method 1 provides a construction of symmetric $1$-designs obtained from the primitive permutation representations. It was first introduced in \cite{keymoo,keymoocorrection} and summarized in \cite[Theorem 5]{Moori}. Notice that the primitive action of $G$ on $\Omega$ implies that the stabilizer $G_{\alpha}$ of a point $\alpha\in\Omega$ is a maximal subgroup of $G$.

\begin{theorem}[Key and Moori]
\label{thm-method1}
Let $G$ be a finite primitive permutation group acting on the set
$\Omega$ of size $n$. Fix an $\alpha \in \Omega$, let $\{\alpha\}\neq \Delta$ be an orbit of the stabilizer $G_{\alpha}$ of
$\alpha$ and $\CB = \{ \Delta^g: \,\, g \in G\}$.
Then $\CD=(\Omega, \CB)$ forms a  $1$-$(n,|\Delta|,|\Delta|)$ design
with $n$ blocks.
\end{theorem}

The second method introduces a technique  from which a large number of non-symmetric $1$-designs could be constructed. 
Denote by $g^G$ the conjugacy class of $g$ in a group $G$. For $M\leq G$, let $M^\times=M-\{1\}$ and denote by $1_M^G$ the permutation character afforded by the action of $G$ on the (left) cosets of $M$ in $G.$ It is known that $|g^G|=|G:C_G(g)|$ and $1_M^G(g)\in \N$. 
A $1$-design is constructed and known as follows \cite[Theorem 12]{Moori}.

\begin{theorem}[Moori]
\label{thm-method2}
Let $G$ be a finite nonabelian simple group, M be a maximal subgroup of $G$ and $g\in M^\times$.
Let $\CP=g^G$ and $\CB=\{(M\cap g^G)^y : y\in G\}$.
Then the structure $\CD=(\CP,\CB)$ is a $1$-$(|g^G|, |M\cap g^G|, 1_M^G(g))$ design.
\end{theorem}

In both methods $G$ acts on design $\CD$ by conjugation. This action is transitive on both points and blocks. If $G$ is simple, it can be embedded into the automorphism groups of the designs constructed by these methods, see Lemma \ref{lem:embedG}. With an ambition to understand more about finite simple groups we study the automorphism groups of these designs. They provide natural structures on which our favorite simple groups act faithfully and transitively. Notice that $G$ acts primitively on both points and blocks of designs constructed by Method 1, but $G$ only acts primitively on blocks of designs constructed by Method 2.

Method 1 has been applied to several sporadic simple groups and finite simple groups of Lie type of small rank by various authors. For example, after applying Method 1 to the Janko simple groups $J_1$ and $J_2$, Key and Moori conjectured that the automorphism group $\Aut(\CD)$ of $\CD$ sits in $\Aut(G).$ However, Rodrigues in his PhD thesis \cite{Rodrigues} showed that the conjecture is not true in general by counterexamples.
Thus, there are simple groups $G$ for which $\Aut(G)$ does not preserve the structure of $\CD$. Here, we explain this behavior in general and  classify all automorphisms of $G$ which can be lifted to $\Aut(\CD)$ and which can not for both methods. The lifted automorphisms form a subgroup of $\Aut(\CD)$, called $AD(G)$ in Theorem \ref{thm:AD1} and Theorem \ref{thm:AD2}.

A quick way to see  the design $(\CP,\CB)$ in Theorem \ref{thm-method2} with those parameters is that there are exactly $1_M^G(g)$ conjugates of $M$ containing $g$ and that the transitive conjugate action of $G$ on $g^G$ fulfils the structure of a design.
With designs $\CD$ constructed by using this method, some results towards the structure of $\Aut(\CD)$ have been carried out in \cite{Moori}, where  $G=\PSL(2,q)$, $M$ is a Borel subgroup of $G$. In these cases, the automorphism groups $\Aut(\CD)$ 
%of those designs $\CD$ have the replication number $\lambda \in \{1,2\}$, and their automorphism groups
are either symmetric groups or wreath products of two symmetric groups. Therefore, we do not have much information on $\Aut(\CD)$ in general.

Recall that if $G$ acts on $\Omega$ and $\Lambda\subset \Omega$ then $G_\Lambda=\Stab_G(\Lambda)=\{g\in G:\Lambda^g=\Lambda\}$.
Here we refine the designs $\CD=(\CP,\CB)$ constructed by Method 2 by looking deeper into their structures, which allows us to define a new design $\CD_I=(\CP_I,\CB_I)$, called {\it reduced design} in Definition \ref{def:ReducedDesign}. The structure of $\CD_I$ is more concrete and more advantageous to study $G$ by the following meaning
\begin{itemize}
\item[i)] $\CD$ and $\CD_I$ have the same replication number and the same number of blocks, but $\CD_I$ has a smaller point set and each block has fewer points;
\item[ii)] $\Aut(\CD_I)\cong \Aut(\CD)/S$ for some normal subgroup $S$ of $\Aut(\CD)$;
\item[iii)] $G$ acts transitively and faithfully on both points and blocks of $\CD_I$, thus $G\leq \Aut(\CD_I)$;
\item[iv)] $\Stab_G(B)$ is maximal in $G$ and conjugate to $M$ for all blocks $B\in \CB_I$;
\item[v)] $|\Stab_G(x)|\leq |\Stab_G(y)|$ for all $x\in \CP$ and $y\in \CP_I$. Thus, the action of $G$ on the point set of $\CD_I$ is closer to being primitive.
\end{itemize}
As a result of this observation we obtain the following.

\begin{theorem}
\label{thm:Stab}
Let $G$ be a finite nonabelian simple group, $M$ a maximal subgroup of $G$ and $g\in M^\times$. Let $\CV:=\{M^x:x\in G \mbox{ such that } g\in M^x\}$, $\displaystyle{A:=\cap_{Q\in \CV}\ Q}$ and $S:=N_G(A)$. Then $C_G(g)\leq S$ and $g^S=A\cap g^G$.
\end{theorem}

In an application to the sporadic Mathieu groups $M_n$ where $n\in \{22,23,24\}$, we are able to retrieve the Steiner system $S(t,t+3,n)$, see Subsection \ref{Mathieu_app}. 
Furthermore, we discover $t$-$(n,t+1,16)$ designs whose automorphism groups are isomorphic to the automorphism groups of those Steiner systems $S(t,t+3,n)$.

The paper is organized as follows. In Section \ref{sec:general}, we prove some general results for both methods: the embedding of $G$ into the automorphism groups of designs and the introduction of reduced designs $\CD_I$. In Section \ref{sec:Method1}, we classify all automorphisms of $G$ which can be lifted to $\Aut(\CD)$ where $\CD$ is a design constructed by Method 1. In Section \ref{sec:Method2}, we study designs $\CD$ constructed by Method 2. First, we classify all automorphisms of $G$ which can be lifted to $\Aut(\CD)$. Next we study the stabilizer in $G$ of a point of the reduced design $\CD_I$, where Theorem \ref{thm:Stab} is proven, and we show some applications on the Mathieu simple groups. In the last section we compute all reduced designs of those constructed by Method 2 for the group $G=\PSL(2,q^2)$, $q$ odd, and its maximal subgroups $M\cong \PGL(2,q)$.

Our notation will be quite standard, and it is as in \cite{ak:book} for designs  and %$\AT$
in \cite{atlas, Hup} for group theory and character theory. For the structure of
finite simple groups and their maximal subgroups we follow the $\AT$ notation \cite{atlas}.

\section{Results applicable for both methods}
\label{sec:general}
Consider the design $\CD$
constructed from a finite simple group $G$ by Method 1 or 2. In this
section  we aim to prove some general results regarding
the structure of $\Aut(\CD)$ and its relation with $\Aut(G).$

\begin{definition}
For each $x\in G$, we define $\td{x}$ acting on $\CD$ as follows:
$u\mapsto u^x$ for all points $u\in \CP=\Omega$, and $B\mapsto
B^x=\{u^x: u\in B\}$ for all blocks $B$ of $\CD$. We set
$\td{G}:=\{\td{x}:x\in G\}$.
\end{definition}

\begin{lemma}
\label{lem:embedG}
$\td{x}\in \Aut(\CD)$ for all $x\in G$. The homomorphism $G\rightarrow
\Aut(\CD), x\mapsto \td{x}$ is one-to-one and $G\cong \td{G}$.
\end{lemma}

\proof
\textbf{Method 1}: Assume that $\CD$ is a $1$-$(|\CP|,
|\Delta|,|\Delta|)$ design constructed by Method 1. If $M=G_\alpha$ and $B=\Delta^g$ for some
$g\in G$, then $B^x=(\Delta^g)^x=\Delta^{gx}\in \mathcal{B}$, it is
clear that $\td{x}\in \Aut(\CD).$ Notice that for all $\beta\in \CP$ and $x\in G$, $G_{\beta^{\td{g}}}=G_\beta^g$ is conjugate to $M$.

Now we show that this embedding is one-to-one by proving that $\td{x}\neq \td{y}$ for
all $x\neq y\in G$. Since $G$ is simple, there exists  $u\in
\CP$ such that $xy^{-1}\not\in G_u,$ because if $xy^{-1}\in
G_u$ for all $u\in \CP,$ then $xy^{-1}\in \bigcap_{u\in\mathcal{P}}G_u=\bigcap_{g\in
G}^{} g^{-1}Mg=\{1_G\}$ which implies $x=y.$ Thus $u^x\neq u^y$,
which implies $\td{x}\neq \td{y}.$

\smallskip
\textbf{Method 2}: Assume that $\CD$ is the design constructed by Method 2 from $g\in M^\times$. Since the points in $\CD$ are a conjugacy class of
$G$ and if a block $B= g^G\cap M$, then $B^x=g^G\cap M^x$, it is
clear that $\td{x}\in \Aut(\CD)$. We shall show that this
embedding is one-to-one by proving that $\td{x}\neq \td{y}$ for all
$x\neq y\in G$.

Since $G$ is simple, there exists  $h\in g^G$ such that
$xy^{-1}\not\in C_G(h),$ because if  $xy^{-1}\in C_G(h)$ for all
$h \in g^G,$ then $xy^{-1}\in Z( \la g^G \ra)=Z(G)=\{1_G\}$ which implies
$x=y.$ Thus $h^x\neq h^y$, which implies $\td{x}\neq \td{y}.$
\eproof

\begin{remark}
{\rm
The proof in Lemma \ref{lem:embedG} for Method 1 also works for a primitive permutation group which acts faithfully on $\Omega$.}
\end{remark}

\begin{definition}
\label{def:Ix+SIx}
For a point $x\in \CP$ of the  $1$-$(v,k,\lambda)$ design $\CD$,
let $B_1,...,B_\lambda$ be all distinct blocks containing $x$. We define
$$I_x:=\cap_{i=1}^\lambda B_i\ .$$ We call $x$ a representative of
the set $I_x$, i.e. $I_x=I_y$ for all $y\in I_x$.

For every $\sigma\in S_{I_x}$, the symmetric group of $I_x$, we
define $\td{\sigma}$ acting on $\CP$ as follows.
$$\td{\sigma}(u):=\left\{\begin{array}{ll} \sigma(u)& \mbox{ if } u\in I_x,\\ u & \mbox{otherwise.}\end{array}\right.$$

Set $\td{S}_{I_x}:=\{\td{\sigma}: \sigma \in S_{I_x}\}$.
\end{definition}

\begin{lemma}
The following hold.
\begin{itemize}
\item[1)] For all $x,y\in \CP$, there is $h\in G$ such that
$\td{h}(I_x)=I_y$;
\item[2)] $\td{\sigma} \in\Aut(\CD)$.
\end{itemize}
\end{lemma}

\proof Since $G$ acts transitively on $\CP$ and
$\td{h}(I_x)=(\cap_{i=1}^\lambda B_i)^h=\cap_{i=1}^\lambda {B_i}^h$
for all $h\in G$, the first claim is clear by the  design
properties. The second claim follows by $\td\sigma(B)=B$ for all
blocks $B$ of $\CD$ and the fact that incidences are preserved. \eproof

\medskip
By the above lemma, $v/|I_x|\in\N$. Let $t:=v/|I_x|$. The points in the  $1$-$(v,k,\lambda)$ design $\CD$
are partitioned into $t$ parts, denoted $\{I_{x_i}:i \in [1,t]\}$. This fact leads to
the following corollary.

\begin{corollary}
\label{cor:SI_construction}
The group $\td{S}_{I_{x_1}}\times ... \times \td{S}_{I_{x_t}}\leq \Aut(\CD)$.
\end{corollary}

\begin{definition}
We set $S(I):=\td{S}_{I_{x_1}}\times ... \times \td{S}_{I_{x_t}}$.
\end{definition}

It is clear that all elements in $S(I)$ fix all blocks of $\CD$, i.e.
$\tau(B)=B$ for all $\tau \in S(I)$ and $B\in\CB$. The converse is considered in the following lemma.

\begin{lemma}
\label{lem:SI_prop}
We have
\begin{itemize}
\item[i)] $S(I)\lhd \Aut(\CD)$.
\item[ii)] If $f\in\Aut(\CD)$ fixes all blocks of $\CD$ then $f\in S(I)$.
\item[iii)]  $S(I)=\cap_{B\in \CB}\ \Stab_{\Aut(\mathcal{D})}(B)$, where
$\CB$ is the block set of design $\CD$.
\end{itemize}
\end{lemma}

\proof i) It suffices to check that $f\tau f^{-1} \in S(I)$ for all
$\tau\in \td{S}_{I_x}$ and $f\in \Aut(\CD)$. It is clear by checking
directly that $f\tau f^{-1} \in \td{S}_{I_{f(x)}}\leq S(I).$
%%i.e. $f\tau f^{-1}(I_{f(x)})=f\tau(f^{-1}(I_{f(x)}))=f\tau(I_x)=f(I_x)=I_{f(x)}$

ii--iii) The claim is clear by the restriction of $f$ to the point set
which is partitioned as
$\mathcal{P}=I_{x_1}\cup...\cup I_{x_t}$.
\eproof

\begin{remark}{\rm
If $\CD=(\CP,\CB)$ is a $1$-$(v,k,\lambda)$ design, from Corollary \ref{cor:SI_construction}, we let $\CP_I:=\{I_{x_i}:i\in [1,t]\}$ %be the set of points
 and 
%the set of blocks 
$\CB_I:=\{\bar{B}:B \in \CB\}$ where $\bar{B}=\{I_{x_i}:I_{x_i}\subset B\}$ for each $B\in\CB$.
Then $(\CP_I,\CB_I)$ forms a 1-design with parameters $1$-$(v/|I_x|,k/|I_x|,\lambda)$.
}\end{remark}

\begin{definition}
\label{def:ReducedDesign}
If $|I_x|>1$, we call the $1$-$(v/|I_x|,k/|I_x|,\lambda)$ design constructed above the \textit{reduced design} $\CD_I$ of $\CD$.

\end{definition}

\begin{remark}
% The theory of $S(I)$ and $\CD_I$ can be applied to any 1-design $\CD$ provided that $\Aut(\CD)$ acts transitively on the point set and the block set.
{\rm
i) By Lemma \ref{lem:SI_prop}, for every $f\in \Aut(\CD)$, if $f\not\in S(I)$ then there exists at least one
$B\in \CB$ such that $f(B)\neq B$.

ii) Since $\Aut(\CD)/S(I)$ acts faithfully on the design $\CD_I$, $\Aut(\CD)/S(I)\leq \Aut(\CD_I)$. It is easy to see that each element of $\Aut(\CD_I)$ can be lifted naturally to an element of $\Aut(\CD)$, which gives the proof of the following theorem.
}
\end{remark}

\begin{theorem}
\label{thm:Aut_DI}
If $\Aut(\CD)$ acts transitively on the point set and the block set, then
 $\Aut(\CD_I)\cong \Aut(\CD)/S(I)$.
\end{theorem}

\section{$\Aut(\CD)$: Method 1}
\label{sec:Method1}

Consider the $1$-$(|\Omega|, |\Delta|,|\Delta|)$ design $\CD$ constructed from a
finite simple group $G$ by using Method 1. In this section  we study the automorphisms of $G$ that could be lifted up to be automorphisms of $\CD$.
Recall that $M=G_{\alpha}$, a maximal subgroup of $G$. Here $\CP=\Omega$ and $\Delta$ can be chosen by $\CP:=\{M^g:g\in G\}$ and $\Delta:=\{M^{g_0g}:g\in M\}$ for some $g_0\in G-M$.

For $\phi\in \Aut(G)$ we define the map $\td{\phi}$  %acting on $\CD$
by $\td{\phi}(u):=\phi(u)$ and $\td{\phi}(B):=\{\phi(v): v\in B\}$ for
all points $u\in \CP$ and blocks $B\in\CB$ of $\CD$.

If $\phi(M)$ is not conjugate to $M$ (in $G$), then $\phi$ can not be lifted to $\Aut(\CD)$ because $\td\phi$ does not preserve the point set $\CP$. For $B = \Delta^y\in \CB$ we have $\td\phi(B)=\td\phi(\Delta)^{\phi(y)}$. To find whether or not that $\td\phi\in \Aut(\CD)$ it suffices to study when $\td\phi(\Delta)\in\CB$.

\begin{lemma}
\label{lem:Med1Aut}
Let $\phi\in\Aut(G)$. The following hold.
\begin{itemize}
\item[i)] If $\phi(M)=M^x$ for some $x\in G$ then either $\td\phi(\Delta)=\Delta^x$ or $\td\phi(\Delta)\cap\Delta^x=\emptyset$;
\item[ii)] $\td\phi\in\Aut(\CD)$ iff $\phi(M)=M^x$ and $\td\phi(\Delta)=\Delta^x$ for some $x\in G$.
\end{itemize}
\end{lemma}
\begin{proof} i) Suppose that $\phi(M)=M^x$ and $\td\phi(\Delta)\cap\Delta^x\neq \emptyset$, we show $\td\phi(\Delta)=\Delta^x$. Since $\Stab_G(\Delta)=M$, we have $\Stab_G(\td\phi(\Delta))=\phi(M)=M^x$. Notice that $\Delta^x=\{M^{g_0gx}:g\in M\}=\{M^{g_0xg^x}:g\in M\}=\{M^{g_0xh}:h\in M^x\}$, so $\Stab_G(\Delta^x)=M^x$. Pick $u\in\td\phi(\Delta)\cap\Delta^x$, we have $\{u^h:h\in M^x\}\subset \td\phi(\Delta)\cap\Delta^x$. Since $M^x$ acts transitively on $\Delta^x$, we obtain $\td\phi(\Delta)=\Delta^x=\{u^h: h \in M^x\}$.

ii) If $\phi(M)=M^x$ and $\td\phi(\Delta)=\Delta^x$ then it is clear that $\td\phi\in\Aut(\CD)$. For the converse, it suffices to show  that  $\phi(M)=M^x$ and $\td\phi(\Delta)\neq\Delta^x$ imply $\td\phi\not\in\Aut(\CD).$ We do this by proving that $\td\phi(\Delta)\not\in\CB$.

Let $T$ be a (right) transversal of $M$ in $G$. Then $\CP=\{M^y:y \in T\}$, the set of all distinct conjugates of $M$ in $G$. Since $\Stab_G(\Delta^y)=M^y$ and $\CB=\{\Delta^y:y\in T\}$, all blocks have pairwise distinct stabilizers. Since $\Stab_G(\td\phi(\Delta))=M^x$ and $\td\phi(\Delta)\neq\Delta^x$, we obtain $\td\phi(\Delta)\not\in\CB$.
\end{proof}

Now we define
$$AD(G):=\{\td\phi: \phi\in\Aut(G) \mbox{ such that } \phi(M)=M^x \mbox{ and } \td\phi(\Delta)=\Delta^x \mbox{ for some } x\in G\}.$$

By Lemma \ref{lem:Med1Aut}, $AD(G)$ is the set of all automorphisms of $G$ that can be lifted to be automorphisms of $\CD$. Clearly we have:

\begin{theorem}
\label{thm:AD1}
$G\cong \td{G}\leq AD(G)\leq \Aut(\CD)$.
\end{theorem}

\begin{remark}
{\rm
From the proof of Lemma \ref{lem:Med1Aut}, if $\phi(M)=M^x$ then $\td\phi(\Delta)$ is always an orbit of $M^x$, so $\td\phi(\Delta)^{x^{-1}}$ is a nontrivial orbit of $M$ (in $\Omega$). The following hold.
\begin{itemize}
\item[i)] If $M$ has only one orbit of size $|\Delta|$, then $\td\phi\in\Aut(\CD)$ iff $\phi(M)=M^x$.
\item[ii)] Suppose that $G$ has only one conjugacy class of maximal subgroups isomorphic to $M$, and let $\mathfrak{S}_{|\Delta|}$ be the set of all orbits of $M$ having the same size $|\Delta|$. Then $\Aut(G)$ acts on $\mathfrak{S}_{|\Delta|}$ by sending $\Delta$ to $\td\phi(\Delta)^{x^{-1}}$ where $\phi(M)=M^x$ for all $\phi\in\Aut(G)$. Clearly the inner automorphism group %$\mbox{Inn}(G)$
    is in the kernel of this action.
    So the outer automorphism group %$\mbox{Out}(G)$
    of $G$ acts on $\mathfrak{S}_{|\Delta|}$.
\item[iii)] If $G$ has only one conjugacy class of maximal subgroups isomorphic to $M$ and $M$ has only one orbit of size $|\Delta|$, then $\Aut(G)\cong AD(G)\leq \Aut(\CD)$.
\end{itemize}
}
\end{remark}

\begin{remark}
{\rm Since $G$ acts primitively on $\CP=\Omega$ and $\CP_I$ is a partition of $\CP$, it is clear that $I_x$ in Definition \ref{def:Ix+SIx} is trivial in Method 1.
}
\end{remark}

\begin{example}
{\rm
i) Let $G=\PSL(2,27)$ and $M\cong D_{26}$ a dihedral group of order $26$. Under the action of $M$ on $\CP$, we have $22$ orbits in which there are $13$ orbits of length $13$ and $8$ orbits of length $26$. The orbits of length $13$ give $1$-$(378,13,13)$ designs. Twelve of these designs have $\Aut(\CD)\cong G$, and the last one has $\Aut(\CD)\cong G{.}6\cong \Aut(G)$. Since $G$ has only one conjugacy class of maximal subgroups isomorphic to $D_{26}$, see \cite[Page 18]{atlas}, and twelve designs have $\Aut(\CD)\cong G$, the outer automorphism group $C_6$ of $G$ acting on $13$ orbits of length $13$ has one fixed point. Therefore, for the last design we have $\Aut(\CD)=AD(G)\cong\Aut(G)$.

ii) Let $G=A_6$ and $M=A_5$ (notice that $G$ has two conjugacy classes of maximal subgroups
isomorphic to $A_5$). The action of $M$ on $\mathcal{P}$ %=\Omega$
 has only two orbits, one of length $1$ and the other of length $5$. The second orbit provides a $1$-$(6,5,5)$ design $\CD_1$, with $\Aut
(\CD_1)=AD(G)=A_6.2_1=S_6$.
Notice that
$\Aut(G)=A_6.2^2$
where $2^2=V_4=\{e,2_1,2_2,2_3\}$, see \cite[Page 4]{atlas}. %with only $2_1$ lifted to $\Aut(\CD_1).$
Only the outer automorphism $2_1$ fixes the conjucgacy class of $M$, the other two involutory outer
automorphisms of $G$ send $M$ to a nonconjugate copy of $M$ in
$G$.

iii) Let $G=A_6$ and $M\cong S_4$ %, a maximal subgroup of $G$ of index $15$
(notice that $G$ has two conjugacy classes of maximal subgroups
isomorphic to $S_4$). Under the action of $M$, $\mathcal{P}$ %=\Omega$
 splits into $3$ orbits whose lengths are $1$, $6$ and $8$. Third orbit of length $8$
provides  a $1$-$(15,8,8)$ design $\CD_2$, with $\Aut (\CD_2)=A_8\gneq
AD(G)=A_6.2_1=S_6$.
Notice that only the outer automorphism $2_1$ fixes the conjugacy class of $M$, the other two
involutory outer automorphisms of $G$ send $M$ to a non-conjugate copy
of $M$ in $G$, see \cite[Page 4]{atlas}.

iv) Let $G=A_9$ and $M \cong \PSL(2,8){:}3$ of index $120$. Under the action of $M$, $\mathcal{P}$ %=\Omega$
 has three orbits whose lengths are $1$, $56$ and $63$. The second orbit of length $56$ provides a $1$-$(120, 56, 56)$ design $\CD$
with $\Aut(\CD) =O_8^+(2){:}2\gneq  AD(G)\cong G$, while
$\Aut(G)=A_9.2=S_{9}$. %\nleqslant \Aut(\CD)$.
%The outer automorphism of $G$, say $\phi$, is of order 2 and $\phi \notin \Aut(\CD)$. Notice that there are two classes of maximal subgroups of $G$ isomorphic to $\PSL(2,8)$, which are interchanged by $\phi$, that is $\phi$ sends $M$ to a non-conjugate copy of $M$ in $G$.
%
The involutory outer automorphism of $G$ is not lifted to $\Aut(\CD)$ since it sends
$M$ to a non-conjugate copy of $M$ in $G$, see \cite[Page 37]{atlas}.

}
\end{example}

\section{$\Aut(\CD)$: Method 2}
\label{sec:Method2}

Consider the $1$-$(|g^G|,|g^G\cap M|,1_M^G(g))$ design $\CD$ constructed from a
finite simple group $G$ and some fixed $g\in M^\times$ by using Method 2. In this section  we prove some general results regarding the structure of
$\Aut(\CD)$ and its relation with $\Aut(G).$

First we study the automorphisms in $\Aut(G)$ that could be naturally lifted to be automorphisms of $\CD$.
For $\phi\in\Aut(G)$, we define the map $\td\phi$ by $\td{\phi}(u):=\phi(u)$ and $\td{\phi}(B):=\{\phi(v): v\in B\}$ for
all points $u\in g^G=\CP$ and blocks $B\in\CB$ of $\CD$. We shall find the conditions on $\phi$ such that $\td\phi\in\Aut(\CD)$. 

The condition $\phi(g)\in g^G$ is a must to preserve the point set $\CP$. So we consider $\phi\in\Aut(G)$ such that $\phi(g)\in g^G$. For a block $B=g^G\cap M^x\in\CB$, we have $\td\phi(B)=\phi(g)^G\cap \phi(M)^{\phi(x)}=g^G\cap \phi(M)^{\phi(x)}$. Thus $\Stab_G(\td\phi(B))=\phi(M)^{\phi(x)}$ by the maximality of $\phi(M)$ in $G$.  Since $\Stab_G(B)=M^y$ for every bock $B=g^G\cap M^y\in\CB$, we obtain the following:

\begin{lemma}
\label{lem:Med2Aut}
For $\phi\in\Aut(G)$, $\td\phi\in \Aut(\CD)$ iff $\phi(g)\in g^G$ and $\phi(M)=M^x$ for some $x\in G$.
\end{lemma}

Now we define
$$AD(G):=\{\td\phi:\phi\in\Aut(G) \mbox{ such that } \phi(g)\in g^G \mbox{ and }\phi(M)=M^x \mbox{ for some } x\in G\}.$$

By Lemma \ref{lem:Med2Aut}, $AD(G)$ is the set of all automorphisms of $G$ that could be lifted to be automorphisms of $\CD$.

\begin{theorem}
\label{thm:AD2}
$G\cong \td{G}\leq AD(G)\leq \Aut(\CD)$.
\end{theorem}

\begin{remark}
{\rm
If $g^G$ is the unique conjugacy class of elements of
order $n$ in $G$ and $G$ has only one conjugacy class of maximal
subgroups isomorphic to $M$, then clearly $\Aut(G)\cong AD(G)\leq
\Aut (\CD).$
}
\end{remark}

\begin{remark}
{\rm
Except for all automorphisms of $G$ that are lifted to automorphisms of $\CD$, the rest of $\phi\in \Aut(G)$, i.e. either $\phi(g^G)\neq g^G$ or $\phi(M)\neq M^x$ for all $x\in G$, can be studied as follows. Suppose that the order $\ord(\phi)=t$. The $1$-design $\CD^*$ with the point set $\cup_{k=1}^t\phi^k(g^G)$ and the block set $\{{\phi}^k(g^G\cap M^x):x\in G, 1\leq k\leq t\}$ receives $\phi \in \Aut(\CD^{*})$.
}
\end{remark}

\begin{lemma}
 $S(I)\cap AD(G)=\{\Id_G\}$.
\end{lemma}

\proof It suffices to show that if $f\in AD(G)$ fixes all blocks
of $\CD$ then $f=\Id_G$. It is clear that $\Stab_G(g^G\cap M)=M$ and
$\Stab_G(B^x)=\Stab_G(B)^x$ for all blocks $B$, $x\in G$. For all
$x,y\in G$ we have $f(B^{xy})=B^{xy}$. Since
$f(B^{xy})=f(B^x)^{f(y)}=(B^x)^{f(y)}$, we have $f(y)y^{-1}\in
\Stab_G(B^x)$ for all $x,y\in G$. Therefore, $f(y)y^{-1}\in\cap_{x\in
G}M^x\unlhd G$. Since $G$ is simple, it forces $f(y)=y$ for all
$y\in G$. Thus, $f=\Id_G.$ \eproof

\begin{remark}
{\rm
If $\lambda=1$ then all blocks are
pairwise disjoint. It is easy to see that $\Aut(\CD)\cong S_k\wr S_b=S_I{:}S_b$ where $b$
is the number of blocks. Thus we are interested in studying the $1$-$(v,k,\lambda)$ designs constructed by Method $2$ with $\lambda\geq 2$.
}
\end{remark}

\begin{example}
\label{ex:PSL29}
{%\em
Apply Method 2 to  $G=\PSL(2,9)\cong A_6$ with $M\cong\PGL(2,3)\cong S_4$ and $g^G =  2A$ is the unique conjugacy class of involutions in $A_6$.
Here $$v=|g^G|=45,\ b=|G/M|=15,\ \lambda= 1_M^G(g) = (1_a+5_a+9_a)(g)=3$$ and $k=(\lambda\times v)/b= 3\times 45/15=9.$

So  we get a  $1$-$(45,9,3)$ design $\CD$. For every point $x\in g^G$
the intersection of the three blocks containing $x$ has size 3.
Therefore,  $\CD_I$ is a symmetric $1$-$(15,3,3)$ design and $S(I)={S_3}^{15}\lhd\Aut(\CD)$. Together with
$AD(G)\cong G{:}2_1\lneq \Aut(G),$ the automorphism group of $\CD$
contains ${S_3}^{15}{:}(G{:}2),$ whose order is exactly the group's order obtained
from Magma. Thus, we have $\Aut(\CD)\cong  {S_3}^{15}{:}S_6$ and
$\Aut (\CD_I)\cong \Aut (\CD)/S(I)\cong G{:}2_1\cong S_6\not\cong \PGL(2,9).$

\medskip
\textit{Detailed computations:} Consider $S_4$ embedded into $A_6$ as follows.
Let $S_4$ act on a size $4$ subset $S$ of $[1,6]$ and let $\{i,j\}=[1,6]-S$.
% be the fixed point set of $S_4$ on $[1,6]$.
If $x\in S_4$ is even, keep it the
same in $A_6$; otherwise, if $x$ is odd, take its image equal to
$x(ij)$. Denote this embedding %of $S_4$ into $A_6$
corresponding to $(ij)$ by $M_{ij}$.
Set $b_{ij}:=(12)(34)^{A_6}\cap M_{ij}$. Then, the intersection of
any pair $b_{ij}\neq b_{lk}$ is known as follows:
$$b_{ij}\cap b_{lk}=\left\{\begin{array}{ll} \{(ij)(rs),(lk)(rs),(ij)(lk)\}& \mbox{ if } %\{i,j\}\cap \{l,k\}=\emptyset \mbox{ and }
\exists\ r,s: \{i,j,l,k,r,s\}=[1,6],\\ \emptyset & \mbox{ otherwise.}\end{array}\right.$$
Fix a block $b_{ij}$, there are ${4 \choose 2}=6$
%$\binom{4}{2}=6$
%$\left({ \begin{array}{l} 4\\2 \end{array} }\right)=6$
blocks having nontrivial intersections with it. Thus, the size of $I_x$ is 3, and a symmetric group $S_3$ acts on each intersection as discussed above. These 15 intersections $I_x$ are transitively acted by $\td{G}\leq AD(G)$.

Notice that the other maximal subgroup $M_2\cong S_4$ of
$G=A_6$ is generated by $(1,2,3)(4,5,6)$,
$(1,4)(3,5)$.
}
\end{example}

\subsection{The stabilizer $\Stab_G(I_x)$ in $G$ of a point $I_x\in \CP_I$ of $\CD_I$.}
Fix $x\in \CP=g^G$ in this subsection. Let $B_i:=M^{g_i}\cap g^G$ for $i\in [1,\lambda]$ be all $\lambda=1_M^G(g)$ distinct blocks containing $x$. We define
$$A_x:=\cap_{i=1}^\lambda\ M^{g_i}.$$

\begin{lemma}
\label{lem:Stab_Ix}
Set $S_x:=\Stab_G(I_x). $ The following hold.
\begin{itemize}
\item[i)] $I_x=A_x\cap g^G$ and $S_x=N_G(A_x)$.
\item[ii)] $C_G(y)A_x\leq S_x$ for all $y\in I_x$.
\item[iii)] Let $H_x:=\la C_G(y):y\in I_x\ra$. Then $H_x$ and $H_xA_x$ are normal in $S_x$.
\item[iv)] $|S_x|=|C_G(x)||I_x|$. Furthermore, $I_x=x^{S_x}$ is a conjugacy class of $S_x$.
\end{itemize}
\end{lemma}

\proof
i) We have $I_x=\cap_{i=1}^\lambda B_i=\cap_{i=1}^\lambda M^{g_i}\cap g^G=A_x\cap g^G$. Since $\{I_y: i\in [1,t]\}$ partition the point set $g^G$ of $\CD$ and preserve the block set, $S_x$ acts %closely
 on $\{B_i: i\in [1, \lambda]\}$. As $g^G$ is invariant under conjugate action, $S_x$ normalizes $A_x$. So $S_x\leq N_G(A_x)$. Since $I_x=A_x\cap g^G$, the set of all elements of $A_x$ conjugate to g in $G$, $I_x$ is invariant under the conjugate action of $N_G(A_x)$. So $N_G(A_x)\leq S_x$.

ii) For $y\in I_x$, by design properties of $\CD$, $\Stab_G(y)=C_G(y)$ acts %closely
 on the set $\{B_i: i\in [1, \lambda]\}$. Thus, $C_G(y)$ normalizes $A_x$ and $C_G(y)A_x\leq S_x=N_G(A_x)$ by part i).

iii) The normality of $H_x$ follows directly from $C_G(y)^a=C_G(y^a)\leq H_x$ for all $y\in I_x$ and $a\in S_x=\Stab_G(I_x)$.

iv) Let $R_x:=\{g_i: i\in [1, n]\}$ be a right transversal of $C_G(x)$ in $G$. For each $y\in g^G$ there exists uniquely $g_i\in R_x$ such that $x^{g_i}=y$. Thus for all $h\in S_x$, we have $I_x^h=I_x$ and there exists uniquely $g_i\in R_x$ such that $h\in C_G(x)g_i$. Therefore $S_x\subset \bigcup_%{g_i\in R_x}^
{x^{g_i}\in I_x} C_G(x)g_i$. % which implies $|S_x|\leq |C_G(x)||I_x|$.
For all $g_i\in R_x$ such that $x^{g_i}\in I_x$, we have $I_x^{g_i}=I_x$ since $G$ acts %closely
 on $\CP_I$. Thus $C_G(x)g_i\subset S_x$ for all $g_i\in R_x$ such that $x^{g_i}\in I_x$, which shows that $S_x=\bigcup_{x^{g_i}\in I_x} C_G(x)g_i$ and $|S_x|=|C_G(x)||I_x|$.
Since $C_G(x)\leq S_x$, $C_{S_x}(x)=C_G(x)$ and $|x^{S_x}|=|S_x/C_G(x)|$. So $x^{S_x}=I_x$.
\eproof

\begin{remark}
{\rm
i) From $S_x=\bigcup_{x^{g_i}\in I_x} C_G(x)g_i$ in the proof of Lemma \ref{lem:Stab_Ix} iv), $\Stab_G(I_x)$ is the maximal subgroup among those $H\leq G$ such that $I_x\subset H$ and $x^H=I_x$.

ii) By Lemma \ref{lem:Stab_Ix} i) $I_x$ is  a union of conjugacy classes in $A_x$, hence $\la I_x\ra\unlhd A_x$. Furthermore, it is clear that $\Stab_G(I_x)=N_G(\la I_x\ra)$.
}
\end{remark}

\subsection{From Mathieu groups to Steiner systems and other $t$-designs}
\label{Mathieu_app}
Here, we shall retrieve the Steiner systems $S(t,t+3,n)$ from the sporadic Mathieu groups $M_n$ for $n=22,23,24$. Furthermore, we discover other $t$-designs whose automorphism groups are isomorphic to the automorphism groups of $S(t,k,n)$. All results are computed by Magma \cite{MAG}.

First, we recall the definition of the dual $\CD^T=(\CP^T,\CB^T)$ of a design $\CD=(\CP,\CB)$ where $\CP^T:=\CB$ and $\CB^T:=\{\beta_x:x\in \CP\}$, where $\beta_x:=\{B\in\CB:x\in B\}$. The results listed in Table \ref{tab:deg_Mathieu} are proceeded by the following two steps:
\begin{itemize}
\item[1.] Obtain reduced designs $\CD_I=(\CP,\CB)$.
\item[2.] Determine the duals  $\CD_I^T=(\CP^T,\CB^T)$ of $\CD_I$. 
\end{itemize}

\begin{table}[!ht]
\caption{Designs constructed from sporadic Mathieu groups.}
\label{tab:deg_Mathieu}
\begin{center}
\begin{tabular}{l|l|l|l|l|l|l}
\hline
\rule{0cm}{0.4cm}
$G$ & $M$ & $\ord(g)$ & $|I_x|$ &
%$\CD_I$ &
$\CD^T_I$ & $\Aut(\CD^T_I)$ & $\Stab_G(b)$
\rule[-0.1cm]{0cm}{0.4cm}\\
\hline
\rule{0cm}{0.4cm}
$M_{24}$ & $M_{23}$ & $2$ & $15$ &
%$1$-$(759,253,8)$ &
$S(5,8,24)$ & $M_{24}$ & $2^4{:}A_8$
\rule[-0.2cm]{0cm}{0.4cm}\\
\hline
\rule{0cm}{0.4cm}
$M_{24}$ & $M_{23}$ & $3$ & $2$ &
%$1$-$(113344,28336,6)$ &
$5$-$(24,6,16)$ & $M_{24}$ & $(3.A_6){:}2$
\rule[-0.2cm]{0cm}{0.4cm}\\
\hline
\rule{0cm}{0.4cm}
$M_{23}$ & $M_{22}$ & $2$ & $15$ &
%$1$-$(253,77,7)$ &
$S(4,7,23)$ & $M_{23}$ & $2^4{:}A_7$
\rule[-0.2cm]{0cm}{0.4cm}\\
\hline
\rule{0cm}{0.4cm}
$M_{23}$ & $M_{22}$ & $3$ & $2$ &
%$1$-$(28336,6160,5)$ &
$4$-$(23,5,16)$ & $M_{23}$ & $(3{\times}A_5){:}2$
\rule[-0.2cm]{0cm}{0.4cm}\\
\hline
\rule{0cm}{0.4cm}
$M_{22}$ & $\PSL(3,4)$ & $2$ & $15$ &
%$1$-$(77,21,6)$ &
$S(3,6,22)$ & $M_{22}.2$ & $2^4{:}A_6$
\rule[-0.2cm]{0cm}{0.4cm}\\
\hline
\rule{0cm}{0.4cm}
$M_{22}$ & $\PSL(3,4)$ & $3$ & $2$ &
%$1$-$(6160,1120,4)$ &
$3$-$(22,4,16)$ & $M_{22}.2$ & $(3{\times} A_4){:}2$
\rule[-0.2cm]{0cm}{0.4cm}\\
\hline
\end{tabular}

where $b\in \CB^T$.
\end{center}
\end{table}

\begin{remark}
{\rm
Notice that $\Stab_G(B)\cong M\cong \Stab_G(y)$ for all $B\in \CB$ and $y\in \CP^T$. Thus in Table \ref{tab:deg_Mathieu} we only list $\Stab_G(b)$ where $b\in \CB^T$. From these computations, $M_n$ acts transitively, but not primitively, on the block set of the $t$-$(n,t+1,16)$ design where $(n,t)\in\{(24,5),(23,4),(22,3)\}$.
}
\end{remark}

Here we provide Magma commands for the computation of $M_{24}$ with the involutory conjugacy class, the others are similar:

\smallskip
\noindent
%%$>$ load pergps; load m24;\\
$>$ load m24; max:=MaximalSubgroups(G);\\
$>$ $\#$max, [$\#$G/max[i]\` { }order : i in [1..$\#$max]];\\
$>$ M:=max[7]\` { }subgroup; cM:=Classes(M); $\#$cM, [cM[i][1] : i in [1..$\#$cM]];\\
$>$ i:=2; a:=cM[i][3]; P:=a $\widehat{{}}$ G; b:= P meet $\{$x : x in M$\}$; $\#$b;\\
$>$ gblox:=b $\widehat{{}}$ G; ca:=[y : y in gblox $|$ a in y]; Ix:=b;\\
$>$ for y in ca do Ix:=Ix meet y; end for; $\#$Ix;\\
$>$ ptsI:=Ix $\widehat{{}}$ G; bloxI:=$\{\}$;\\
$>$ for x in gblox do\\
for$>$ cx:=[y subset x : y in ptsI]; bx:=$\{\}$;\\
for$>$ for s in [1..$\#$cx] do if cx[s] then bx:=bx join $\{$s$\}$; end if; end for;\\
for$>$ bloxI:=bloxI join $\{$bx$\}$;\\
for$>$ end for;\\
$>$ desI:=Design$<$1,$\#$ptsI$|$bloxI$>$; desI;\\
$>$ dB:=$\{\}$;\\
$>$ for x in [1..$\#$ptsI] do\\
for$>$ cx:=[x in y : y in bloxI]; bx:=$\{\}$;\\
for$>$ for s in [1..$\#$cx] do if cx[s] then bx:=bx join $\{$s$\}$; end if; end for;\\
for$>$ dB:=dB join $\{$bx$\}$;\\
for$>$ end for;\\
$>$ desdB:=Design$<$5,$\#$bloxI$|$dB$>$; desdB;\\
$>$ aut:=AutomorphismGroup(desdB);\\
$>$ "Aut(DIT)"; ChiefFactors(aut); "StabG(b)"; ChiefFactors(Stabilizer(aut,bx));

\section{Designs of $G=\PSL(2,q^2)$ and $M\cong \PGL(2,q)$ for odd $q$ }
\label{sec:SL2q}
Let $G=\PSL(2,q^2)$,  where $q$ is a power of an odd prime $p$. By \cite[Proposition 4.5.3]{Kleidman}, $G$ has two conjugacy classes of maximal subgroups
isomorphic to $\PGL(2,q)$. These two maximal subgroups correspond to two conjugacy classes of unipotent elements in $G$, which are called squared and non-squared. We denote these two subgroups  by $M_1$ and $M_2$ respectively, %both are isomorphic to $\PGL(2,q)$,
where $M_1$ contains squared unipotent and $M_2$ contains non-squared unipotent elements.

Here we study the designs constructed by Method 2  with $G=\PSL(2,q^2)$, $M=M_i$ for $i=1,2$ and all $g\in M^\times$. By Theorem \ref{thm-method2}, they are $1$-$(|g^G|,|g^G\cap M|,1_M^G(g))$ designs. Since $|G|=\frac{q^2(q^4-1)}{2}$, $|M_i|=q(q^2-1)$, we have $|G:M|=\frac{q(q^2+1)}{2}$.

The following properties of $G$ and $M_i$ are well-known and they could be obtained easily from computation:
\begin{itemize}
\item[i)] An element of $M_i$ is either unipotent or semisimple.
\item[ii)] There is exactly one unipotent conjugacy class in $M_i$, which is either squared or non-squared in $G$. 
\item[iii)] $G$ has exactly one conjugacy class $y^G$ of involutions and $C_G(y)\cong D_{q^2-1}$, a dihedral group of order $q^2-1$. Each $M_i$ has exactly two involutory conjugacy classes $x^{M_i}$ and $C_{M_i}(x)\cong D_{2(q\pm 1)}$.
\item[iv)] The permutation character $1_{M_i}^G$ has been evaluated as follows. For $g\in G$,
$$1_{M_i}^G(g)=\left\{\begin{array}{cl} |G:M_i|&\mbox{ if $g=1$,}
\\q & \mbox{ if $\ord(g)=2$ or $p$},
\\ \frac{q+1}{2} &\mbox{ if $\ord(g)|\ \frac{q-1}{2}$ and $\ord(g)\gneq 2$,}
\\ \frac{q-1}{2} &\mbox{ if $\ord(g)|\ \frac{q+1}{2}$ and $\ord(g)\gneq 2$,}
\\ 0 & \mbox{ otherwise.}
\end{array}\right.$$
\end{itemize}

\medskip
Now we apply these results for each conjugacy with  $G=\PSL(2,q^2)$ and
$M=M_1.$ %=\PGL(2,q)$.
The dual method will give simlar results for $M=M_2$.

\subsection{Designs from involutions $g\in M$}
\label{subsec:involution}

Recall that $G$ has exactly one conjugacy class $g^G$ of involutions, $C_G(g)\cong D_{q^2-1}$, and
$|g^G|= \frac{q^2(q^2+1)}{2}$. Moreover, $C_G(g)$ is known as the normalizer of a maximal split torus of $G$.

Here $M$ has exactly two
involutory conjugacy classes, call them $x_1^M$ and $x_2^M$, where $C_M(x_1)\cong D_{2(q+1)}$ and $C_M(x_2)\cong D_{2(q-1)}$. Moreover, $C_M(x_1)$ and $C_M(x_2)$ are also recognized as the normalizers of two maximal tori, non-split and split respectively. %The only nontrivial element in $C_G(x_i)$ is $x_i$
As the orders of dihedral groups $C_G(x_i)$ are divisible by 4, we have $Z(C_G(x_i))=\{1,x_i\}$.
We have $M\cap g^G=x_1^M\cup
x_2^M$ and
$$|M\cap
g^G|=|x_1^M|+|x_2^M|=\frac{|M|}{2(q+1)}+\frac{|M|}{2(q-1)}=\frac{q(q-1)}{2}+\frac{q(q+1)}{2}=q^2.$$
So the design constructed by Method 2 is a 1-$(\frac{q^2(q^2+1)}{2},q^2,q)$ design, call it $\CD$.

\begin{lemma}
For all $x\in g^G$ we have $|I_x|=1$  if $q>3$, and $|I_x|=3$ if $q=3$.
\end{lemma}
%% Notice that when $q=3$ then $D_{2(q-1)}$ is abelian.

\proof
For $q=3$, Example \ref{ex:PSL29} was demonstrated the computation in details. Now we suppose that $q>3$.

Let $x\in g^G$ and $B_i:=M^{g_i}\cap g^G$ for $i\in [1,q]$ be all $q$ distinct blocks containing $x$. Recall that $C_G(x)\cong D_{q^2-1}$. First of all, we shall show that $C_G(x)$ acts on the set $\{B_i:i \in [1,q]\}$ into two orbits of sizes $\frac{q-1}{2}$ and $\frac{q+1}{2}$.

If $x\in M^{g_i}$ such that $C_{M^{g_i}}(x)\cong D_{2(q+1)}$ then, since $g^G$ is invariant under conjugate action, we have
$$\begin{array}{ll}
\Stab_{C_G(x)}(B_i)&=\Stab_{C_G(x)}(M^{g_i})\\
&=\Stab_G(M^{g_i})\cap C_G(x)\\
&=M^{g_i}\cap C_G(x)=C_{M^{g_i}}(x).\end{array}$$
So $B_i$ belongs to an orbit of size $\frac{q-1}{2}$. Using the same argument, if $C_{M^{g_i}}(x)\cong D_{2(q-1)}$, then $B_i$ belongs to an orbit of size $\frac{q+1}{2}$.
As $M$ has only two conjugacy classes of involutions, an orbit of the action of $C_G(x)$ on $\{B_i:i \in [1,q]\}$ has size either $\frac{q-1}{2}$ or $\frac{q+1}{2}$.
Since $q>3$, $C_G(x)$ acts on $q$ blocks with no invariant points. %it cannot happen with two orbits having the same size.
 Thus our claim holds.
% Since there are $q$ blocks and there is no invariant block under the action of $C_G(x)$,

Let $A:=\cap_{i=1}^q\ M^{g_i}$. By Lemma \ref{lem:Stab_Ix}, $C_G(x)A$ is a nontrivial proper subgroup of $G$. Since $C_G(x)\cong D_{q^2-1}$ is maximal in $G$, $C_G(x)A$ must equal $C_G(x)$, which implies $A\leq C_G(x)$. Since $I_x=\cap_{i=1}^q\ B_i=A\cap g^G$ and the above argument works for all elements in $I_x$, all elements in $I_x$ commute.

%When $q=3$, $D_{2(q+1)}=D_{q^2-1}$, Example \ref{ex:PSL29} was demonstrated the computation in details.
Suppose that there is $y\in I_x-\{x\}$. Notice that $y\not\in Z(C_G(x))=\{1,x\}$ since  $C_G(x)\cong D_{q^2-1}$. Since $I_x\subset A\leq C_G(x)$ and $C_G(x)$ acts on $I_x$ by conjugation, the conjugacy class of $y$ in $C_G(x)$ is also contained in $I_x$. However, all these elements do not commute in $D_{q^2-1}$ when $q>3$. This completes the proof.
\eproof

\subsection{Designs from nontrivial  unipotents $g\in M$}
\label{subsec:unipotent}

Recall that $g\in M$ is in the squared unipotent conjugacy
class of $G=\PSL(2,q^2)$.  We have $C_G(g)\in \Syl_p(G)$ and
$C_M(g)\in \Syl_p(M)$, thus
$$|g^G|=|G/C_G(g)|=\frac{q^4-1}{2} \mbox{ and } |g^M|=|M/C_M(g)|=q^2-1.$$

Since $M$ has a unique nontrivial unipotent conjugacy class,
$u^G\cap M=u^M$. Hence, we obtain a 1-$(\frac{q^4-1}{2},q^2-1,q)$ design, call it $\CD$.

\begin{lemma}
$|I_x|=q-1$ for all $x\in g^G$. Moreover, $I_x=C_M(x)\cap g^G$ for $x\in M\cap g^G$ and the reduced design $\CD_I$ is a 1-design with parameters $(\frac{(q^2+1)(q+1)}{2},q+1,q)$.
\end{lemma}

\proof
Let $x\in g^G$ and $B_i:=M^{g_i}\cap g^G$ for $i\in [1,q]$ be all $q$ distinct blocks containing $x$. Recall that
$C_G(x)\cap M^{g_i}=C_{M^{g_i}}(x)\in \Syl_p(M^{g_i})$. First we claim that
the action of $C_G(x)$ on $\{B_i\}_{i=1}^q$ is transitive. Since $C_G(x)\not\subset M^{g_i}$, the action is nontrivial.
We have $\Stab_{C_G(x)}(M^{g_i})=C_G(x)\cap \Stab_G(M^{g_i})=C_G(x)\cap M^{g_i}=C_{M^{g_i}}(x)$. Thus the orbit of $M^{g_i}$ has size $|C_G(x)/C_{M^{g_i}}(x)|=q$, which proves the claim.

Since $M^{g_i}$ has a unique unipotent conjugacy class, $C_{M^{g_i}}(x)^\times\subset B_i$. Since $C_G(x)$ is abelian, $C_{M^{g_i}}(x)$ is invariant under this action. So $C_{M^{g_i}}(x)^\times \subset B_i^h$ for all $h\in C_G(x)$. Thus $C_{M^{g_i}}(x)^\times\subset I_x$.

Let $A:=\cap_{i=1}^q M^{g_i}$. Since $C_{M^{g_i}}(x)\leq A$ and $C_{M^{g_i}}(x)\in \Syl_p(M^{g_i})$, we have $C_{M^{g_i}}(x)\in \Syl_p(A)$. Suppose that there is $y\in I_x-C_{M^{g_i}}(x)$. Using the same argument, we obtain $C_{M^{g_i}}(y)^\times \subset I_x$ and $C_{M^{g_i}}(y)\in \Syl_p(A)$. Since $A$ contains at least two distinct Sylow $p$-subgroups and $A\lneq M^{g_i}\cong \PGL(2,q)$, by \cite[Corollary 2.2 and Corollary 2.3]{O.King}, we obtain $A\cong\PSL(2,q)$. Thus $A\lhd M^{g_i}$ for all $i\in[1,q]$ since $|M^{g_i}/A|=2$. This implies that $M^{g_i}\leq N_G(A)$ for all $i$. However, the maximality of $M^{g_i}$ and $q\geq 3$ imply $N_G(A)=G$,  which contradicts the simplicity of $G$. Thus there is no $y\in I_x-C_{M^{g_i}}(x)$ and this completes the proof.
\eproof

\begin{remark}
{\rm
In the above proof,  $|C_G(x)A|=%|C_G(x)||A|/|C_G(x)\cap A|=
\frac{q^2(q^2-1)}{2}$ and $\Stab_G(I_x)=N_G(C_{M^{g_i}}(x))$ is a Borel subgroup of $G$, which is maximal in $G$. By Lemma \ref{lem:Stab_Ix}, $\Stab_G(I_x)=C_G(x)A$. Thus, together with the maximality of $\Stab_G(I_x)$, $G$ acts primitively on both points and blocks of $\CD_I$.
}
\end{remark}

\subsection{Designs from semisimples $g\in M$ with $\ord(g)> 2$}
\label{subsec:semisimple}

Since $G=\PSL(2,q^2)$, $g$ is regular and belongs to a maximal split
torus $C_G(g)$ of order $\frac{q^2-1}{2}$ in $G$. Thus, $|g^G|=q^2(q^2+1)$.

Since $\gcd(q-1,q+1)=2$, if $\ord(g)|(q-1)$ then $g$ belongs to a maximal torus of order $q-1$ in $M$ and $|g^M|=q(q+1)=|g^G\cap M|$; otherwise if $\ord(g)|(q+1)$ then $g$ belongs to a maximal torus of order $q+1$ in $M$ and $|g^M|=q(q-1)=|g^G\cap M|$.

So we obtain 1-$(q^2(q^2+1),q(q\pm 1),\frac{q\pm 1}{2})$ designs, call them $\CD^\pm$ respectively.

\begin{lemma}
$I_x=\{x,x^{-1}\}$ for all $x\in g^G$. Then the reduced designs $\CD^\pm_I$ are 1-designs with parameters $(\frac{q^2(q^2+1)}{2},\frac{q(q\pm 1)}{2},\frac{q\pm 1}{2})$ respectively.
\end{lemma}

\proof
It suffices to prove this lemma for the 1-$(q^2(q^2+1),q(q+ 1),\frac{q+1}{2})$ design $\CD^+$. Let $x\in g^G$ and $B_i:=M^{g_i}\cap g^G$ for $i\in [1,\frac{q+1}{2}]$ be all the $\frac{q+1}{2}$ distinct blocks containing $x$. Since $x$ is regular and contained in the maximal torus $T_i:=C_{M^{g_i}}(x)$, its inverse $x^{-1}$ is also in $T_i$. Since the normalizer $N_{M^{g_i}}(T_i)\cong D_{2(q-1)}$ fuses $x$ to $x^{-1}$, we have both $x,x^{-1}\in B_i$. Thus $x,x^{-1}\in I_x$.

Consider the conjugate action of $T:=C_G(x)$ on $\{B_i:i\in [1, \frac{q+1}{2}]\}$. Recall that $T$ is a maximal split torus of order $\frac{q^2-1}{2}$ in $G$.  We have
$$\Stab_T(B_i)=\Stab_T(M^{g_i})=\Stab_G(M^{g_i})\cap T=M^{g_i}\cap T=T_i$$ of order $q-1$. Thus $T$ acts transitively on $\{B_i:i\in [1, \frac{q+1}{2}]\}$.
%Notice that $M^{x_i}\cap x^G=x^{M^{x_i}}$ and $x^{M^{x_i}}\cap C_G(x)=\{x,x^{-1}\}$.

Let $A:=\cap_{i=1}^{\frac{q+1}{2}}\ M^{g_i}$. By Lemma \ref{lem:Stab_Ix}, $TA\lneq G$. By the simplicity of $G$ and \cite[Corollary 2.2]{O.King}, $TA\leq N_G(T)=:N\cong D_{q^2-1}$, a maximal subgroup of $G$. Since the Weyl group $N$ of $T$ controls the fusion of all regular elements of $T$ in $G$ and the Lie rank of $G$ is one, we have $g^G\cap N=\{x,x^{-1}\}$. Therefore, $A\cap g^G\subset N\cap g^G=\{x,x^{-1}\}$ which completes the proof.
\eproof

\section*{Acknowledgement}
We would like to thank the referees for their corrections and suggestions. As a result our paper has been improved significantly.

\end{document}